\documentclass[12pt]{article}
\usepackage[utf8]{inputenc}
\usepackage{amsmath}
\usepackage[english]{babel}

\usepackage[a4paper,top=3cm,bottom=2cm,left=3cm,right=3cm,marginparwidth=1.75cm]{geometry}

\usepackage{amsmath}
\usepackage{graphicx}
\usepackage[colorinlistoftodos]{todonotes}
\usepackage[colorlinks=true, allcolors=blue]{hyperref}
\usepackage{amssymb}
\usepackage{amssymb}
\usepackage{amsthm}
\usepackage{authblk}
\usepackage{amsmath}
\usepackage{amsfonts}
\usepackage{amssymb}
\usepackage{amscd}
\usepackage{latexsym}
\usepackage{breqn}

\theoremstyle{plain}
\newtheorem{thm}{Theorem}[section]
\newtheorem{lem}[thm]{Lemma}
\newtheorem{prop}[thm]{Proposition}
\newtheorem{cor}[thm]{Corollary}

\theoremstyle{definition}
\newtheorem{conj}{Conjecture}[section]
\newtheorem{exmp}{Example}[section]

\theoremstyle{remark}
\newtheorem*{rem}{Remark}

\title{Toeplitz kernels and the backward shift}
\author{Ryan O'Loughlin \\ E-mail address: \href{mailto:mm1rol@leeds.ac.uk}{mm12rol@leeds.ac.uk}}
\affil{School of Mathematics, University of Leeds, Leeds, LS2 9JT, U.K.}
\begin{document}
\maketitle
\begin{abstract}
    In this paper we study the kernels of Toeplitz operators on both the scalar and the vector-valued Hardy space for $ 1 < p < \infty $. We show existence of a minimal kernel for any element of the vector-valued Hardy space and we determine a symbol for the corresponding Toeplitz operator. In the scalar case we give an explicit description of a maximal function for a given Toeplitz kernel which has been decomposed in to a certain form. In the vectorial case we show not all Toeplitz kernels have a maximal function and in the case of $p=2$ we find the exact conditions for when a Toeplitz kernel has a maximal function. For both the scalar and vector-valued Hardy space we study the minimal Toeplitz kernel containing multiple elements of the Hardy space, which in turn allows us to deduce an equivalent condition for a function in the Smirnov class to be cyclic for the backward shift.

\end{abstract}

\section{Introduction}
The purpose of this paper is to study the kernels of Toeplitz operators on vector-valued Hardy spaces. In particular we shall address the question of whether there is a smallest Toeplitz kernel containing a given element or subspace of the Hardy space. This will in turn show how Toeplitz kernels can often be completely described by a fixed number of vectors, called maximal functions. We also discover an interesting and fundamental link between this topic and cyclic vectors for the backward shift.

Throughout we will generally work in the vector-valued Hardy space of functions analytic in the unit disc. Throughout the paper we will fix  $1 < p < \infty$. The vector-valued Hardy space is denoted $H^p(\mathbb{D},\mathbb{C}^n)$ and is defined to be a column vector of length $n$, where each entry takes values in $H^p$. Background theory on the classical Hardy space $H^p$ can be found in \cite{duren1970theory, nikolski2002operators}.

Let $P$ be the pointwise Riesz projection $L^p(\mathbb{D}, \mathbb{C}^n) \to H^p(\mathbb{D},\mathbb{C}^n)$. For $G$ a $n$-by-$n$ matrix with each entry of $G$ taking values in $L^{\infty}(\mathbb{T})$ ($\mathbb{T}$ is the unit circle) the Toeplitz operator on the space $H^p(\mathbb{D},\mathbb{C}^n)$, with symbol $G$, is defined by 
$$
T_G(f) = P(Gf).
$$
We have the direct sum decomposition $L^p(\mathbb{D}, \mathbb{C}^n) =  \overline{H^p_0}(\mathbb{D},\mathbb{C}^n) \oplus H^p(\mathbb{D},\mathbb{C}^n)$, where $\overline{H^p_0}(\mathbb{D},\mathbb{C}^n):= \{\overline{f}: f \in H^p(\mathbb{D},\mathbb{C}^n), f(0) = 0 \}$, and hence $f \in \ker T_G$ iff $Gf \in \overline{H^p_0}(\mathbb{D},\mathbb{C}^n)$. Perhaps the most fundamental example of Toeplitz kernels is the class of model spaces; for $\theta$ an inner function, the kernel of the Toeplitz operator $T_{\overline{\theta}}$ is denoted $K_{\theta}$ and can be verified to be $H^p \cap \theta \overline{H^p_0}.$ It follows from Beurling's Theorem that the model spaces are the non-trivial closed invariant subspaces for the backward shift operator on $H^p$. 

The concept of nearly backward shift invariant subspaces was first introduced by Hitt in \cite{hitt1988invariant} as a generalisation to Hayashi's results concerning Toeplitz kernels in \cite{MR853630}. These spaces were then studied further by Sarason \cite{sarason1988nearly}. The idea of near invariance was then generalised and it was first demonstrated in \cite{MR3286053} that for any inner function $\eta$, Toeplitz kernels are nearly $\overline{\eta}$ invariant. This means if $\phi$ is in the kernel of a Toeplitz operator $T_G: H^p(\mathbb{D},\mathbb{C}^n) \to H^p(\mathbb{D},\mathbb{C}^n)$ and $\phi$ is such that $\overline{\eta} \phi \in H^p(\mathbb{D},\mathbb{C}^n)$, then $\overline{\eta} \phi \in \ker T_G$. This observation shows not all $\phi \in H^P(\mathbb{D}, \mathbb{C}^n)$ lie in one-dimensional Toeplitz kernels. In the scalar case (i.e when $n=1$) Theorem 5.1 in \cite{MR3286053} shows the existence of a Toeplitz kernel of smallest size containing $\phi \in H^p$ (formally known as the minimal kernel for $\phi$ and denoted $\kappa_{min}(\phi)$), furthermore a Toeplitz operator $T_g$ is defined such that $\kappa_{min}(\phi) = \ker T_g$. This motivates our study for section 2 where we address the question: is there a minimal Toeplitz kernel containing a given element $\phi \in H^p(\mathbb{D}, \mathbb{C}^n )$?

In section 3 we consider the following parallel question: given any Toeplitz kernel $K$ does there exist a $\phi$ such that $K = \kappa_{min}(\phi)$? We call such a $\phi$ a \textit{maximal function} for $K$. It has been shown in \cite{MR3286053} that in the scalar Toeplitz kernel case,  whenever the kernel is non-trivial there does exist a maximal function. Theorem 3.17 in \cite{camara2018scalar} shows that for $p=2$ every matricial Toeplitz kernel which can be expressed as a fixed vector-valued  function multiplied by a non-trivial scalar Toeplitz kernel also has a maximal function. The results of section 3 show not all non-trivial matricial Toeplitz kernels have a maximal function and for $p=2$ we find the exact conditions for when a Toeplitz kernel has a maximal function. An interesting application of the study of maximal functions is given in \cite{MR3806717}, which fully characterises multipliers between Toeplitz kernels in terms of their maximal functions.

Section 5 of \cite{MR3286053} asks if there is a minimal Toeplitz kernel containing a closed subspace $E \subseteq H^p(\mathbb{D}, \mathbb{C}^n) $, so in sections 4 and 5 we turn our attention to finding the minimal kernel of multiple elements $ f_1 \hdots f_k \in H^p(\mathbb{D}, \mathbb{C}^n)$. This in turn allows us to find the minimal Toeplitz kernel containing a finite-dimensional space $E$, as we can set $E = $span$\{f_1 \hdots f_k \}$. When considering scalar Toeplitz kernels previous results considering the minimal kernel for multiple elements have been presented in \cite{camara2016model}, in particular Theorem 5.6 of \cite{camara2016model} shows that when $\kappa_{min}(f_j) =  K_{\theta_j}$ for some inner function $\theta_j$ then $\kappa_{min}(f_1 \hdots f_j) = K_{LCM(\theta_1, \hdots \theta_j)}$. The corollaries of section 4 show the fundamental link between the minimal kernel of two elements in $H^p$ and the cyclic vector for the backward shift, in fact we deduce an equivalent condition for a function to be cyclic for the backward shift on $N^+$.

\subsection{Notations and convention}
\begin{itemize}
    \item For a function $f \in H^p$ we write $f = f^i f^o$, where $f^i / f^o$ is an inner/outer factor of $f$ respectively. Furthermore the inner/outer factor of $f$ is unique up to a unimodular constant.
    \item The backward shift is denoted $B$, and is defined by $B(f) = \frac{f-f(0)}{z}$, where $f(z)=\sum_{n=0}^{\infty}a_n z^n$ is an analytic function in the unit disc.
    \item Toeplitz operators on $H^p$ will be referred to as scalar Toeplitz operators and Toeplitz operators on $H^p(\mathbb{D},\mathbb{C}^n)$ will be referred to as matricial Toeplitz operators.
    \item $GCD$ stands for greatest common divisor and for two inner functions $I_1, I_2$ $I_1|I_2$ will denote that $I_1$ divides $I_2$.
    \item Throughout, all Toeplitz operators will be bounded, and hence have bounded symbols.
\end{itemize}

\section{Minimal kernel of an element in $H^p(\mathbb{D}, \mathbb{C}^n )$}
For $G$ a $n$-by-$n$ matrix symbol we say $\ker T_G$ is the minimal kernel of 
$ \begin{pmatrix}
\phi_1 &
\hdots &
\phi_n
\end{pmatrix}^T
$ if $\begin{pmatrix}
\phi_1 &
\hdots &
\phi_n
\end{pmatrix}^T \in \ker T_G$, and if $\begin{pmatrix}
\phi_1 &
\hdots &
\phi_n
\end{pmatrix}^T \in \ker T_{H}$ for any other $n$-by-$n$ matrix symbol $H$ we have $ \ker T_{G} \subseteq \ker T_{H}$.

Section 5.1 in \cite{MR3286053} addresses whether there always exists a minimal Toeplitz kernel containing a function in $H^p(\mathbb{D}, \mathbb{C}^n)$. A complete answer to this question was not given, however Theorem 5.5 showed the existence of a minimal Toeplitz kernel containing any rational $\phi$ in $H^p(\mathbb{D}, \mathbb{C}^n)$. The main result of this section is to give a complete answer to this question. We will show existence of a minimal Toeplitz kernel containing any $\phi \in H^p(\mathbb{D}, \mathbb{C}^n)$, and define an operator $T_G$ such that $\kappa_{min}(\phi) = \ker T_G$.

\begin{lem}\label{outerconstruct}
 For any $\phi_1 \hdots \phi_n \in H^p$ there exists an outer function $u$ such that $|u| = |\phi_1| + \hdots + |\phi_n| + 1$.
\end{lem}
\begin{proof}
Outer functions have a representation 
$$
u(z) = \alpha  \exp \left( \frac{1}{2\pi} \int_0^{2\pi} \frac{e^{it} + z}{e^{it} - z} k(e^{it}) dt \right),
$$
where $|\alpha|=1$, $k \in L^1(\mathbb{T})$ is real. Moreover $k = \log|u|$.\

In the above representation, if we let $k = \log(|\phi_1| + \hdots + |\phi_n| + 1)$ it then follows that $|u| = |\phi_1| + \hdots + |\phi_n| + 1$. It can be seen that $k = \log(|\phi_1| + \hdots + |\phi_n| + 1) \in L^1(\mathbb{T})$, as $0 < \log(1+x) < x $ for all $x >0 $, and $\phi_1 \hdots \phi_n \in L^1$.
\end{proof}

We say $f$ belongs to the Smirnov class, denoted $N^+$, if $f$ is holomorphic in the disc and
$$
\lim_{r \to 1^-} \int_{\mathbb{T}} \log(1+|f(rz)|) dm(z) = \int_{\mathbb{T}} \log(1+|f(z)|) dm(z) < \infty,
$$
where $dm$ is the normalised Lebesgue measure on $\mathbb{T}$, i.e $dm = d\theta / 2 \pi$. On $N^+$ the metric is defined by 
$$
\rho(f,g) = \int_{\mathbb{T}} \log(1+|f(z) - g(z)|) dm(z).
$$

Let $\log L$ denote the class of complex measurable functions $f$ on $\mathbb{T}$ for which $\rho(f,0) < \infty$. One can check $\log L$ is an algebra. Furthermore section 3.6.3 of \cite{cima2000backward} along with the argument laid out on p. 122 of Gamelin's book \cite{gamelin1984uniform} shows that when $\log L $ is equipped with $\rho$ as a metric $\log L $ is a topological algebra ($f_n \to f$ and $g_n \to g$ in $\log L $ $\implies f_n + g_n \to f + g$ and $f_n g_n \to fg$ in $\log L $). Proposition 3.6.10 in \cite{cima2000backward} further shows that $N^+$ is the closure of the analytic polynomials in $\log L $, and hence $N^+$ is a topological algebra.

Throughout various literature there have many equivalent ways to define the Smirnov class; for the sake of completeness we list these in the following proposition.

\begin{prop}\label{quotientouter}
The following three statements are equivalent \begin{enumerate}
    \item $f \in N^+$.
    \item $f \in \{  \frac{f_1}{f_2} : f_2\text{ is outer } f_1, f_2 \in H^{\infty}  \}.$
    \item $f \in \{  \frac{f_1}{f_2} : f_2\text{ is outer } f_1, f_2 \in H^{1/2}  \}.$
    \item $f = b s_{\mu_1} f^o$, where $b$ is a Blaschke product, $s_{\mu_1}$ a singular inner function with respect to the measure $\mu_1$ and $f^o$ an outer function.
\end{enumerate}
\end{prop}
\begin{proof}

Following the argument laid out in the proof of Theorem 2.10 \cite{duren1970theory} shows the equivalence of 1 and 4. $2 \implies 3$ is immediate. $3 \implies 4$ follows from the fact that the reciprocal of an outer function is outer and so is the product of two outer functions. We now show 4 implies 2 to show all the statements are equivalent.

We construct two outer functions $F_1, F_2$ such that $|F_1| = {min(1, |f|)}$, and $|F_2| = {min(1, |f|^{-1})}$, as in Lemma \ref{outerconstruct} we only need to prove that $\log({min(1, |f|)})$ and $\log({min(1, |f|^{-1})})$ are in $L^1$ in order to do this. Using the radial limits of $f$ we can define $f$ a.e on $\mathbb{T}$, we define $E := \{ z \in \mathbb{T} : |f(z)|>1 \}$ and $F := \{ z \in \mathbb{T} : |f(z)| \leqslant 1 \}$. Then 
$$
\int_{\mathbb{T}}  \log(min(1, |f|))  = \int_{E} \log(min(1, |f|)) + \int_{F} \log(min(1, |f|)) = 0 + \int_{F} \log|f|.
$$ As $|f|$ is log integrable over the whole of $\mathbb{T}$ it is also log integrable over any subset of $\mathbb{T}$, so the expression above shows $\log {min(1, |f|)} \in L^1$. A similar computation shows $min(1, |f|^{-1})$ is log integrable, it then follows that $F_1, F_2 \in   H^{\infty}$. As $|{F_2}||f| = |{F_1}|$ a.e on $\mathbb{T}$, taking outer factors we can conclude $f^o = \frac{F_1}{F_2}$ so $f = \frac{b s_{\mu_1} F_1}{F_2}$.
\end{proof}
Notice from the fourth characterisation of $N^+$ from above, that if $f \in N^+$ and the boundary function is in $L^p$, then $ f \in H^p$, i.e., $N^+ \cap L^p = H^p$.

\vskip 0.5cm
We present the main Theorem of this section.
\begin{thm}\label{kmin}
Let $u$ be an outer function such that $|u| = |\phi_1| + \hdots + |\phi_n| + 1$, where $ \phi_1 \hdots \phi_n \in H^p$, then 
$$
\kappa_{min}
\begin{pmatrix}
\phi_1 \\
\vdots \\
\phi_n
\end{pmatrix} = \ker T_{
\begin{pmatrix}
\overline{\phi_1} \overline{z} / \phi_1^o & 0 & \hdots & \hdots & \hdots & \hdots & 0 \\
- \phi_2 / u & \phi_1 / u & 0 & \hdots & \hdots & \hdots &  0  \\
- \phi_3 / u & 0 & \phi_1 / u & 0 & \hdots & \hdots &  0  \\
- \phi_4 / u & 0 & 0 & \phi_1 / u & 0 & \hdots & 0  \\
\vdots & \vdots & \vdots & \vdots & \vdots & \vdots &  \vdots  \\
- \phi_n / u & 0 & \hdots & \hdots & \hdots & 0 &  \phi_1 / u \\
\end{pmatrix}} .
$$
\end{thm}
\begin{proof}
We denote the above symbol by $G$. We have to show that if $\begin{pmatrix}
\phi_1 &
\hdots &
\phi_n
\end{pmatrix}^T \in \ker T_H$, for any bounded $n$-by-$n$ matrix $H$, then every $\begin{pmatrix}
f_1 &
\hdots &
f_n
\end{pmatrix}^T \in \ker T_G
$
also lies in $\ker T_H$. To this end let $\begin{pmatrix}
\phi_1 &
\hdots &
\phi_n
\end{pmatrix}^T \in \ker T_H $, then
$$
\begin{pmatrix}
h_{11} & h_{12} & \hdots & h_{1n} \\
h_{21} & h_{22} & \hdots & h_{2n} \\
\vdots & \vdots & \vdots & \vdots \\
h_{n1} & h_{n2} & \hdots & h_{nn}
\end{pmatrix}
\begin{pmatrix}
\phi_1 \\
\vdots \\
\vdots \\
\phi_n
\end{pmatrix}
 =
 \begin{pmatrix}
\overline{zp_1} \\
\vdots \\
\vdots \\
\overline{zp_n}
\end{pmatrix},$$
for some $p_1 \hdots p_n \in H^p$, so that $\phi_1 h_{i1} + \phi_2 h_{i2} + \hdots + \phi_n h_{in} = \overline{zp_i} $ for each $ i \in \{ 1 \hdots n \}$. Let $ \begin{pmatrix}
f_1 &
\hdots &
f_n
\end{pmatrix}^T \in \ker T_{G}$, then $f_1 = \frac{\phi_1 \overline{p}}{\overline{\phi_1^o}}
$ for some $p \in H^p$. Rows $2$ to $n$ of $G$ take values in $N^+ \cap L^{\infty} = H^{\infty}$, so from row $i \in \{ 2 \hdots n \}$ in $G\begin{pmatrix}
f_1 &
\hdots &
f_n
\end{pmatrix}^T \in \overline{H^p_0( \mathbb{D}, \mathbb{C}^n)}$, taking in to account that $H^p \cap \overline{H^p_0} = \{ 0 \}$, we deduce
$$
f_1 \frac{\phi_i}{u} = f_i \frac{\phi_1}{u}.
$$
Substituting our value for $f_1$ we can write $f_i$ as,
$$
f_i = \frac{\phi_i \overline{p} }{\overline{\phi_1^o}},
$$
so
$$
\begin{pmatrix}
{f_1} \\
\vdots \\
{f_n}
\end{pmatrix} = 
 \frac{\overline{p} }{\overline{\phi_1^o}}
\begin{pmatrix}
\phi_1 \\
\vdots \\
\phi_n
\end{pmatrix},
$$
and hence
$$
H\begin{pmatrix}
{f_1} \\
\vdots \\
{f_n}
\end{pmatrix} = H 
\begin{pmatrix}
\phi_1 \\
\vdots \\
\phi_n
\end{pmatrix}\frac{\overline{p} }{\overline{\phi_1^o}} = \begin{pmatrix}
\overline{zp_1} \\
\vdots \\
\overline{zp_n} 
\end{pmatrix}\frac{\overline{p} }{\overline{\phi_1^o}}.
$$
Proposition \ref{quotientouter} shows $\overline{zp_i} \frac{\overline{p}}{\overline{\phi_1^o}} \in \overline{zN^+} \cap L^p = \overline{H^p_0}$, so we conclude
$$
\begin{pmatrix}
{f_1} \\
\vdots \\
{f_n}
\end{pmatrix} \in \ker T_H.
$$
\end{proof}

\begin{rem}
The above symbol for the minimal kernel is not unique. In fact we can show there are at least $n$ different symbols (not including permuting the rows of the symbol) which represent the same kernel, each depending on the minimal kernel in the scalar case, of $\phi_j ,$ where $ j \in \{ 1 \hdots n \}$. Consider the symbol
$$
\begin{pmatrix}
0 & \hdots & \hdots & \hdots & \hdots & \hdots & \overline{\phi_j} \overline{z} / \phi_j^o & 0 & \hdots &   0 \\
0 & \phi_j / u & 0 & \hdots & \hdots & \hdots &  - \phi_2 / u & 0 &  \hdots    &  0  \\
0 & 0 & \phi_j / u & 0 & \hdots & \hdots & - \phi_3 / u & 0  &  \hdots & 0  \\
0 & 0 & 0 & \phi_j / u & 0 & \hdots & - \phi_4 / u & 0 & \hdots &   0  \\
\vdots & \vdots & \vdots & \vdots & \vdots & \vdots &  \vdots &  \vdots &  \vdots &  \vdots    \\
 \phi_j / u & 0 & \hdots & \hdots & \hdots & 0 &  - \phi_1 / u & 0 & \hdots & 0  \\
0 & \hdots & \hdots & \hdots & \hdots & \hdots & \phi_{j+1} / u & \phi_j / u & 0  & 0 \\ 
\vdots & \vdots & \vdots & \vdots & \vdots & \vdots &  \vdots &  \vdots &  \vdots &  \vdots    \\
0 & \hdots & \hdots & \hdots & \hdots & \hdots & \phi_n / u & 0 & \hdots  & \phi_j / u \\
\end{pmatrix},
$$
where the first non zero entry on the first row is in the $j$'th column, and the row where the first entry is non zero is the $j$'th row. This can also be checked to be a symbol for the minimal kernel.
\end{rem}

\section{Maximal functions for $\ker T_G$}
Recall that $\phi$ is a maximal function for a Toeplitz kernel $K$ if $\kappa_{min}(\phi) = K$. Unlike the scalar case not all matricial Toeplitz kernels have a maximal function. A simple explicit example to show this is 
$$
\ker T_{\begin{pmatrix}
\overline{z} & 0 \\
0 & \overline{z} \\
\end{pmatrix}} = \{
\begin{pmatrix}
\lambda \\
\mu \\
\end{pmatrix} : \lambda,  \mu \in \mathbb{C} \}.
$$
Suppose some fixed $\begin{pmatrix}
\lambda_1 \\
\mu_1 \\
\end{pmatrix} \in \mathbb{C}^2$ give a maximal function, then $$\begin{pmatrix}
\lambda_1 \\
\mu_1 \\
\end{pmatrix} \in \ker T_{\begin{pmatrix}
\mu_1 & -\lambda_1 \\
0 & 0 \\
\end{pmatrix}} = \begin{pmatrix}
\lambda_1 \\
\mu_1 \\
\end{pmatrix} H^p ,$$ but 
$$
\ker T_{\begin{pmatrix}
\overline{z} & 0 \\
0 & \overline{z} \\
\end{pmatrix}} \not\subseteq  \ker T_{\begin{pmatrix}
\mu_1 & -\lambda_1 \\
0 & 0 \\
\end{pmatrix}},
$$
so $
\ker T_{\begin{pmatrix}
\overline{z} & 0 \\
0 & \overline{z} \\
\end{pmatrix}} $ can not have a maximal function.

We can build on this example to give a condition for when Toeplitz kernels do not have a maximal function.

We use the notation $\ker T_G(0):= \{ f(0) : f \in \ker T_G \} $. For a matrix $A$ with each entry of $A$ being a holomorphic function in the disc we write $A(0)$ to mean $A$ with each entry evaluated at $0$.
\begin{thm}\label{negative max}
If $\ker T_G$ is such that  $\dim \ker T_G(0) > 1$ then $\ker T_G$ does not have a maximal function.
\end{thm}
\begin{proof}
Suppose for contradiction $\ker T_G$ is such that  $\dim \ker T_G(0) > 1$ and $\ker T_G$ has a maximal function
$$
v = \begin{pmatrix}
v_1 \\
\vdots \\
v_n
\end{pmatrix},
$$
then for any symbol $H$ if $v \in \ker T_H $, we must have $\ker T_G \subseteq \ker T_H$. Let
$$
x = \begin{pmatrix}
x_1 \\
\vdots \\
x_n
\end{pmatrix},
y = \begin{pmatrix}
y_1 \\
\vdots \\
y_n
\end{pmatrix},
$$
be two linearly independent vectors in $\ker T_G(0)$. Pick $i , j \in \{ 1 ... n \}, i < j $ such that 
$$
\begin{pmatrix}
0 \\
\vdots \\
x_i \\
0 \\
\vdots \\
0 \\
x_j \\
0 \\
\vdots \\
0
\end{pmatrix}, \begin{pmatrix}
0 \\
\vdots \\
y_i \\
0 \\
\vdots \\
0 \\
y_j \\
0 \\
\vdots \\
0
\end{pmatrix},
$$
span a two dimensional subspace of $\mathbb{C}^n$.
Let $n$ be the largest integer such that $\frac{v_i}{z^n}$ and $\frac{v_j}{z^n}$ lie in $H^p$, let $u$ be an outer function such that $|u| = |v_i| + |v_j| +1$, and let $$
H = \begin{pmatrix}
0 & \hdots & 0 & \frac{v_j}{z^n u } & 0 & \hdots & 0 & - \frac{v_i}{z^n u } & 0 & \hdots & 0 \\
0 & \hdots & 0 & 0 & 0 & \hdots & 0 & 0 & 0 & \hdots & 0 \\
\vdots & \vdots & \vdots & \vdots & \vdots & \vdots & \vdots & \vdots & \vdots & \vdots \\
0 & \hdots & 0 & 0 & 0 & \hdots & 0 & 0 & 0 & \hdots & 0 \\ 
\end{pmatrix},
$$where the first non zero entry is in the i'th column and the second is in the $j$'th column. As $\frac{v_i}{z^n u}, \frac{v_j}{z^n u} \in L^{\infty} \cap N^+ = H^{\infty}$ each entry of $H$ takes values in $H^{\infty}$, furthermore $v \in \ker T_H$, so 
$$
\ker T_G \subseteq \ker T_H,
$$
which means that
$$
\ker T_G(0) \subseteq \ker T_H(0).
$$
For $\begin{pmatrix}
f_1 &
\hdots &
f_n
\end{pmatrix}^T \in \ker T_H$ we have $f_i \frac{v_j}{z^n u} = f_j\frac{v_i}{z^n u} $, and by dividing $v_i, v_j$ by $z^n$, we have ensured there is a linear relation between $f_i(0)$ and $f_j(0)$. So the $i$'th and $j$'th coordinate of $\ker T_H(0)$ only span a one dimensional subspace of $\mathbb{C}^n$, but we have picked $ i, j$ so that the $i$'th and $j$'th coordinate of $\ker T_G(0)$ span a two dimensional subspace of $\mathbb{C}^n$, which is a contradiction. So we conclude that maximal functions do not exist whenever  $ \dim \ker T_G(0) > 1 $.
\end{proof}

We now aim to generalise Dyakonov's decomposition of Toeplitz kernels which is Theorem 1 in \cite{dyakonov2000kernels} to a matrix setting, we will then use this result to further study maximal functions. In the case of $p=2$ Theorem 7.4 of \cite{MR2736342} presents a similar formula to what we will obtain.

We define $L^{\infty}( \mathcal{L}(\mathbb{C}^n))$ to be the space of all $n$-by-$n$ matrices with each entry taking values in $L^{\infty}$, we analogously define $H^{\infty}( \mathcal{L}(\mathbb{C}^n))$ and $N^{+}( \mathcal{L}(\mathbb{C}^n))$. For a $n$-by-$n$ matrix inner function $I$ we denote the $B$-invariant subspace $\ker T_{I^*} = I (\overline{H^p_0(\mathbb{D}, \mathbb{C}^n}) \cap H^p(\mathbb{D}, \mathbb{C}^n)$ by $K_{I}$. $K_{I}$ can easily be checked to be $B$-invariant by noting if $ I^* {f} \in \overline{H^p_0(\mathbb{D}, \mathbb{C}^n}) $, then $I^* f(0) \in \overline{H^p (\mathbb{D}, \mathbb{C}^n})$ and so $I^* ({f - f(0)}) \in \overline{H^p(\mathbb{D}, \mathbb{C}^n})$, which implies $ I^* \frac{f - f(0)}{z} = I^* {B(f)} \in \overline{H^p_0(\mathbb{D}, \mathbb{C}^n})$.

For a symbol $G$, if $\det G$ is an invertible element in $L^{\infty}$ then $\int_{\mathbb{T}} \log{\frac{1}{|\det G(z)|} } dm(z) < \infty$, and so $\int_{\mathbb{T}} \log{{|\det G(z)|} } dm(z)  = -\int_{\mathbb{T}} \log{\frac{1}{|\det G(z)|} } dm(z) > - \infty$. This means we can construct a scalar outer function $q$ such that $| \det G | = |q| $.
\begin{lem}
Let $G \in L^{\infty}( \mathcal{L}(\mathbb{C}^n))$ be such that $\det G$ is invertible in $L^{\infty}$ and let $q$ be the outer function such that $| \det G | = |q| $. Then if we define $G^' \in  L^{\infty}( \mathcal{L}(\mathbb{C}^n))$ to be the matrix $G$ with the first row divided by $\overline{q}$, we have $\ker T_G = \ker T_{G^'}$. Furthermore $\det G^'$ is unimodular.
\end{lem}
\begin{proof}
We only need to consider the first row of $G^'$. Denote the first row of $G$ (respectively $G^'$) by $G_1$ (respectively $G_1^')$. As $q$ is invertible in $H^{\infty}$, for $f \in H^p(\mathbb{D},\mathbb{C}^n)$ we have $G_1 f \in \overline{H^p_0}$ if and only if $\frac{G_1}{\overline{q}} f \in \overline{H^p_0}$. The fact $\det G^'$ is unimodular is a result of linearity of the determinant in each row.
\end{proof}
Under the assumption that $\det G$ is invertible in $L^{\infty}$, by the argument laid out above we can assume without loss of generality that $\det G$ is actually unimodular. Theorem 4.2 of \cite{barclay2009solution} states we can now write $G$ as \begin{equation}\label{main}
    G = {G_2}^* {G_1} ,
\end{equation} with $G_1, G_2 \in H^{\infty}( \mathcal{L}(\mathbb{C}^n))$. Furthermore taking the determinant of our unimodular $G$ shows us that $1 = | \det{ G_2^*} | | \det{ G_1} |$ and so $\det{ G_2^*}$ and $\det{ G_1}$ are invertible in $L^{\infty}$, and it then follows ${ G_2^*}$ and $G_1$ are invertible in $ L^{\infty}( \mathcal{L}(\mathbb{C}^n))$.

By \eqref{main} under the assumptions above we can write $f \in \ker T_G$ if and only if $f \in H^p(\mathbb{D},\mathbb{C}^n)$ and $G_2^* G_1 f \in \overline{H^p_0(\mathbb{D},\mathbb{C}^n)}$ i.e $G_1 f \in \ker T_{G_2^*}$. However the following proposition shows the kernel of $T_{G_2^*}$ can be simplified.

\begin{prop}
If $G_2 \in H^{\infty}( \mathcal{L}(\mathbb{C}^n))$ then $\ker T_{G_2^*} = \ker T_{(G_2^i)^*}$.
\end{prop}
Before we begin the proof we make a remark about inner-outer matrix factorisation. Definition 3.1 in \cite{katsnelson1997theory} of outer functions in $N^{+ }( \mathcal{L}(\mathbb{C}^n))$ is that $E \in N^{+}( \mathcal{L}(\mathbb{C}^n))$ is outer if an only if $\det E$ is outer in $N^+$. Theorem 5.4 of \cite{katsnelson1997theory} says that given a function $F \in N^{+}( \mathcal{L}(\mathbb{C}^n)) $ such that $\det{F}$ is not equal to the 0 function, there exist matrix functions $F^i$ inner and $F^o$ outer (respectively $F^{i'} , F^{o'}$) such that we may write $F$ as $F = F^i F^o $ (respectively $F = F^{o'} F^{i'} )$.

\begin{proof}
Since $\det (G_2^o)$ is outer in $H^{\infty}$ and invertible in $L^{\infty}$, it is invertible in $H^{\infty}$, so $ (G_2^o)^*  $ is invertible in $\overline{H^{\infty}( \mathcal{L}(\mathbb{C}^n))}$. Then, after writing $G_2$ as $G_2 = G_2^i G_2^o$, it immediately follows that $\ker T_{G_2^*} = \ker T_{(G_2^i)^*}$.

\end{proof}
The following theorem is the generalisation of Dyakonov's decomposition of Toeplitz kernels to a matrix setting.
\begin{thm}\label{dyakonov}
Using the decomposition for $G$ given in \eqref{main}, $$\ker T_G =  (G_1^{i'})^* \left( (G_1^{o'})^{-1} K_{G_2^i} \cap G_1^{i'} H^p(\mathbb{D},\mathbb{C}^n) \right).$$
\end{thm}
\begin{proof}
Using the proposition above and \eqref{main} we may write $f \in \ker T_G$ if and only if $ f \in H^p(\mathbb{D},\mathbb{C}^n)$ and $G_1 f \in K_{G_2^i}$. We write $G_1 = G_1^{o'} G_1^{i'}$. Since $\det G_1^{o'}$ is outer in $H^{\infty}$ and invertible in $L^{\infty}$, it is invertible in $H^{\infty}$, which means $(G_1^{o'})^{-1} \in H^{\infty}( \mathcal{L}(\mathbb{C}^n))$. Hence we can write the condition $ f \in H^p(\mathbb{D},\mathbb{C}^n)$ and $G_1 f \in K_{G_2^i}$ as $G_1^{i'} f \in (G_1^{o'})^{-1}K_{G_2^i} \cap G_1^{i'} H^p(\mathbb{D},\mathbb{C}^n) $ and so $f \in \ker T_G$ if and only if
$$
 f \in  (G_1^{i'})^* \left( (G_1^{o'})^{-1} K_{G_2^i} \cap G_1^{i'} H^p(\mathbb{D},\mathbb{C}^n) \right).
$$
\end{proof}

\begin{prop}
Let $K$ be a $B$-invariant subspace of $H^p(\mathbb{D}, \mathbb{C}^n)$ such that $K$ evaluated at 0 is a one-dimensional subspace of $\mathbb{C}^n$. Then $K$ is of scalar type i.e, $K$ is a fixed vector multiplied by a scalar backward shift invariant subspace of $H^p$.
\end{prop}
\begin{proof}
Let $K$ evaluated at 0 be equal to the span of $\begin{pmatrix}
\lambda_1 &
\hdots &
\lambda_n
\end{pmatrix}^T$, then by assumption for any $f \in K$ we must have $f(0) = x_0 \begin{pmatrix}
\lambda_1 &
\hdots &
\lambda_n
\end{pmatrix}^T$ for some $x_0 \in \mathbb{C}$. Similarly $B(f)(0) = x_1 \begin{pmatrix}
\lambda_1 &
\hdots &
\lambda_n
\end{pmatrix}^T$ for some $x_1 \in \mathbb{C}$, and we repeat this process recursively to obtain $B^i(f)(0) = x_i \begin{pmatrix}
\lambda_1 &
\hdots &
\lambda_n
\end{pmatrix}^T$. Noting that $B^i(f)(0)$ is the coefficient of $z^i$ for $f$ and polynomials are dense in $H^p$, we deduce that $f = \sum_{i=0}^{\infty}x_i z^i \begin{pmatrix}
\lambda_1 &
\hdots &
\lambda_n
\end{pmatrix}^T$. Furthermore $$\{ \sum_{i=0}^{\infty}x_i z^i \in H^p : \sum_{i=0}^{\infty}x_i z^i \begin{pmatrix}
\lambda_1 &
\hdots &
\lambda_n
\end{pmatrix}^T \in K \}$$ is $B$-invariant because $K$ is.
\end{proof}

\begin{cor}
If $\ker T_G(0)$ is a one-dimensional subspace of $\mathbb{C}^n$ and in the decomposition of the kernel given in Theorem \ref{dyakonov} we have ${G_1^{i'}} = I$ and $K_{G_2^i}$ is non-trivial, then $\ker T_G$ has a maximal function.
\end{cor}
\begin{proof}
If $G_1^{i'} = I$ then we have $G_1^{o'} \ker T_G = K_{G_2^i}$, so $G_1^{o'} (0)\ker T_G (0) = K_{G_2^i} (0)$. Which means either $K_{G_2^i} (0)$ is a one-dimensional subspace of $\mathbb{C}^n$ or is equal to $0$, but as $K_{G_2^i}$ is $B$-invariant it can never be the case that $K_{G_2^i} \subseteq z H^p( \mathbb{D}, \mathbb{C}^n)$. So we must have $K_{G_2^i}(0)$ is a one-dimensional subspace of $\mathbb{C}^n$. Then by the previous proposition $K_{G_2^i}$ must be equal to $\begin{pmatrix}
\lambda_1 &
\hdots &
\lambda_n
\end{pmatrix}^T K_{\theta}$ for some $\begin{pmatrix}
\lambda_1 &
\hdots &
\lambda_n
\end{pmatrix}^T \in \mathbb{C}^n$, and some scalar inner function $\theta$.

We now use Corollaries \ref{max existence} and \ref{max function decomp} which are proved later in section 4 but the proof is independant of any previous results. If we let $m$ be the maximal function of $K_{\theta}$ (which exists by Corollary \ref{max existence}) then by Corollary \ref{max function decomp} given any $f \in \ker T_G$ we can write $ f = (G_1^{o'})^{-1}\begin{pmatrix}
\lambda_1 &
\hdots &
\lambda_n
\end{pmatrix}^T m \overline{n}$ for some $n \in N^+$. So if $(G_1^{o'})^{-1} \begin{pmatrix}
\lambda_1 &
\hdots &
\lambda_n
\end{pmatrix}^T m \in \ker T_H$ then
$$
Hf = H (G_1^{o'})^{-1} \begin{pmatrix}
\lambda_1 \\
\vdots \\
\lambda_n \\
\end{pmatrix} m \overline{n} \in \overline{n H^p(\mathbb{D}, \mathbb{C}^n)} \subseteq \overline{N^+(\mathbb{D}, \mathbb{C}^n)}.
$$
Furthermore $f \in H^p(\mathbb{D}, \mathbb{C}^n)$ and $H$ is bounded so we must actually have $Hf \in \overline{H^p(\mathbb{D}, \mathbb{C}^n)}$, and so $f \in \ker T_H$. As our $f$ was arbitrary we have $\ker T_G \subseteq \ker T_H$. This shows $(G_1^{o'})^{-1} \begin{pmatrix}
\lambda_1 &
\hdots &
\lambda_n
\end{pmatrix}^T m$ is a maximal vector for $\ker T_G$.
\end{proof}

For $1<p< \infty$ and a Toeplitz operator $T_g: H^p \to H^p$  Theorem 2 in \cite{hartmann2003extremal} shows existence of an extremal function $q \in \ker T_g$, and an inner function $I$ vanishing at $0$ such that: \begin{enumerate}
    \item If $p \leqslant 2$ then $q K_{I}^2 \subset \ker T_g \subset q K_{I}^p$.
    \item If $p \geqslant 2$ then $qK_{i}^p \subset \ker T_g \subset qK_{I}^2$.
\end{enumerate}
We now state a reformulation of this result, which may be viewed as a generalisation of the result given by Hayashi in \cite{MR853630} to $1 < p < \infty$.
\begin{cor}\label{max}
\begin{enumerate}
    \item If $p \leqslant 2$ then $ \ker T_g = q K_{I}^p \cap H^p$.
    \item If $p \geqslant 2$ then $\ker T_g = qK_{I}^2 \cap H^p$.
\end{enumerate}
\end{cor}
\begin{proof}
We will prove statement $(1)$. The $\subset$ inclusion is clear from the original result. To show the other inclusion we first observe that as $q K_{I}^2 \subset \ker T_g$ we must have $q I \overline{z}  \in \ker T_g$. Then for all $p \in H^p$ we must then have 
$$
g q I \overline{zp} \in \overline{z N^+},
$$
and so if $q I \overline{zp}$ also lies in $H^p$ we must have $g q I \overline{zp} \in \overline{H^p_0}$, and so consequently $q I \overline{zp} \in \ker T_g$. The result now follows from the fact that any element of $q K_{I}^p \cap H^p$ can be written as $q I \overline{zp} $ for some $p \in H^p$.
\end{proof}

Although the existence of maximal functions in the scalar case has been established in \cite{MR3286053}, we can use the above corollary to give an alternate expression for a maximal function of a given scalar Toeplitz kernel.

\begin{cor}
If $\ker T_g$ is expressed as in Corollary \ref{max}, then $\kappa_{min}(q I\overline{z}) = \ker T_g$.
\end{cor}
\begin{proof}
We will prove the statement in the case $p \leqslant 2$. It is clear $qI \overline{z} \in \ker T_g$. If $qI\overline{z} \in \ker T_h$ for any other bounded symbol $h$, then for any $\overline{p} \in \overline{H^p}$ such that $q I \overline{zp} \in H^p$, because $h q I \overline{z} \in \overline{H^p_0}$, we must have $h q I \overline{zp} \in \overline{H^p_0}$.
\end{proof}
\subsection{Maximal functions when $p=2$}
For the remainder of this section we set $p=2$, and only consider Toeplitz operators 
$$
T_G :  H^2(\mathbb{D}, \mathbb{C}^n) \to H^2(\mathbb{D}, \mathbb{C}^n).
$$
\vskip 0.5cm
When considering whether a given Toeplitz kernel has a maximal function the space $W:= \ker T_G \ominus (\ker T_G \cap \ zH^2(\mathbb{D}, \mathbb{C}^n )) $ is central to this problem. We know from Corollary 4.5 in \cite{MR2651921} that $\ker T_G$ can be written as
\begin{equation}\label{ker}
\ker T_G = [ W_1, W_2, ... W_r] (H^2(\mathbb{D}, \mathbb{C}^r) \ominus \Phi H^2(\mathbb{D}, \mathbb{C}^{r'})),
\end{equation}
where $W_1, ... W_r$ is an orthonormal basis for $W$, $\Phi$ is a $r $ by $ r^{'}$ matrix inner function vanishing at 0 (i.e an isometry from $ (H^2(\mathbb{D}, \mathbb{C}^{r'})$ to $(H^2(\mathbb{D}, \mathbb{C}^r)$) and in this decomposition $ r^{'} \leqslant r$.

\begin{lem}\label{scalar type}
 $ \dim \ker T_G(0) = \dim W$.
\end{lem}
\begin{proof}
 $W_1(0), ... W_r(0)$ are linearly independent, as if $W_k(0) = \sum_{i \neq k} \lambda_i W_i(0)$ this would mean $W_k - \sum_{i \neq k} \lambda_i W_i$ vanishes at 0 and therefore lies in $zH^2(\mathbb{D}, \mathbb{C}^n )$. Next we show that  $W_1(0), ... W_r(0)$ span $\ker T_G(0)$. Evaluating $\ker T_G$ at 0 gives
$$
\ker T_G(0) = [ W_1(0), W_2(0), ... W_r(0)] \mathbb{C}^r ,
$$
which is equal to the span of $W_1(0), ... W_r(0)$. So $W_1(0), ... W_r(0)$ are a basis for $\ker T_G(0)$.
\end{proof}

Taking in to account Theorem \ref{negative max} and the previous lemma we can conclude if $\ker T_G$ is such that $\dim W > 1$, then $\ker T_G$ does not have a maximal function. This leaves us with the following question: if $\ker T_G$ is such that $\dim W=1$ does this Toeplitz kernel have a maximal function? When $\dim W=1$, using the Sarason style decomposition from \eqref{ker} we can write
\begin{equation}\label{decomp}
\ker T_G = W_1 (H^2 \ominus \Phi H^2),
\end{equation}
where $\Phi$ is an inner function vanishing at 0 or $\Phi = 0$. So either:
\begin{enumerate}
    \item $\ker T_G = W_1 K_{\Phi}$,
    \item $\ker T_G = W_1 H^2$.
\end{enumerate}
In case 1 $K_{\Phi}$ is a Toeplitz kernel so $\ker T_G$ has a maximal function given by $W_1 \Phi \overline{z}$ as shown in Theorem 3.17 in \cite{camara2018scalar}.

For case 2 we find that unlike the scalar Toeplitz kernel case there are non-trivial matricial Toeplitz kernels that are shift invariant, for example $\ker T_{\begin{pmatrix}
1 & -1 \\
0 & 0
\end{pmatrix}} = \begin{pmatrix}
1 \\
1
\end{pmatrix}H^2$. In case 2, $\ker T_G$ can not have a maximal function as if it did we would have $\kappa_{min}(\phi) = W_1 H^2$, but this can't be the case as Theorem \ref{kmin} shows the minimal kernel of any element $\phi \in H^2(\mathbb{D}, \mathbb{C}^n)$ is not shift invariant (in particular $\phi z \not\in \kappa_{min}(\phi))$.
We can summarise these results to conclude the following theorem.

\begin{thm}
A non zero Toeplitz kernel, $\ker T_G$, has a maximal function if and only if both: $\dim W=1$ (or equivalently $\dim \ker T_G(0) = 1)$), and when $\ker T_G$ is decomposed as in \eqref{decomp}, $\ker T_G $ takes the form $\ker T_G = W_1 K_{\Phi}$.
\end{thm}
\begin{rem}
From \eqref{decomp} these two conditions can be concisely written as, $\dim \ker T_G(0) = 1$ and $\ker T_G$ is not shift invariant.
\end{rem}
\begin{proof}
Lemma \ref{scalar type} and Theorem \ref{negative max} show that if $\dim W > 1$ then $\ker T_G$ does not have a maximal vector. Conversely if $\dim W =1$ then writing the kernel as we have done in \eqref{decomp} shows that if $\ker T_G$ can be written as $ W_1 K_{\Phi}$, then it necessarily must have a maximal function. If $\ker T_G = W_1 H^2$, then $\ker T_G$ can have no maximal function as minimal kernels are never shift invariant.
\end{proof}
\begin{cor}
If $\ker T_G$ is non zero, then $\ker T_G$ is of scalar type if and only if $\dim \ker T_G(0) = 1$.
\end{cor}
\begin{proof}
If $\ker T_G$ is of scalar type it is clear that $\dim \ker T_G(0) = 1$. Conversely if $\dim \ker T_G(0) = 1$, then Lemma \ref{scalar type} shows $\dim W = 1$, and then \eqref{decomp} shows $\ker T_G$ is of scalar type.
\end{proof}
\begin{thm}
If $\Phi = 0$ in \eqref{decomp} i.e if $\ker T_G = W_1 H^2$, then for any $f \in H^2$ which is a cyclic vector for the backward shift on $H^2$, we have $\kappa_{min}(W_1, W_1f) = W_1 H^2 $.
\end{thm}
\begin{proof}
It is clear the two vectors are in the required kernel. Theorem 4.4 in \cite{MR2651921} shows that multiplication by $W_1$ is an isometric mapping from $H^2$ to $H^2(\mathbb{D}, \mathbb{C}^2)$, so $W_1$ is a $2$ by $1$ matrix inner function.  If $ W_1,  W_1f \in \ker T_H$ for any bounded $H$, then for any $\lambda \in \mathbb{C}$
$$
W_1(f - \lambda) \in \ker T_H.
$$
So setting $\lambda = f(0)$, and using the near invariance property of Toeplitz kernels we see that 
$$
 W_1 \frac{f- f(0)}{z} =  W_1 B(f) \in \ker T_H.
$$
Repeating this inductively gives $W_1 B^n(f) \in \ker T_H$ for all $n \in \mathbb{N}$, and as $f$ is cyclic for the backward shift and $W_1$ is inner, we can deduce
$$
W_1H^2 \subseteq \ker T_H.
$$
\end{proof}
This demonstrates that the number of maximal functions needed to specify a matricial Toeplitz kernel is highly non-trivial and poses the question: for an arbitrary Toeplitz kernel $\ker T_G$, how large should $k$ be such that we can find $ \phi_1 \hdots \phi_k$ where $\kappa_{min}(\phi_1 \hdots \phi_k ) = \ker T_G$? In this case we call  $ \phi_1 \hdots \phi_k$ a \textit{maximal $k$-tuple of functions} or when $k=2$ a \textit{maximal pair of functions} for $\ker T_G$.

We examine the case further for $n=2$. We have seen if $\dim W=2$ then $\ker T_G$ does not have a maximal function, however we will now show if $\dim W=2$ under certain conditions $\ker T_G$ does have a maximal pair of functions. For a matrix $A$ we denote $C_i(A)$ to be the $i$'th column of $A$.

\begin{prop}
If the decomposition of $\ker T_G$ in \eqref{ker} is such that $\Phi$ is square i.e. if $\ker T_G = [ W_1, W_2] (H^2(\mathbb{D}, \mathbb{C}^2) \ominus \Phi H^2(\mathbb{D}, \mathbb{C}^{2}))$, then $\ker T_G$ has a maximal pair of functions given by 
$[W_1,W_2]C_1(\Phi \overline{z})$, and $[W_1,W_2]C_2(\Phi \overline{z})$.
\end{prop}

\begin{proof}
When $\Phi$ is square we have $\Phi \Phi^{*} = \Phi^{*} \Phi = I$, and so a computation shows $(H^2(\mathbb{D}, \mathbb{C}^2) \ominus \Phi H^2(\mathbb{D}, \mathbb{C}^{2})) = \ker T_{\Phi^*}$. Then it is clear both vectors are in the required kernel.

Take any $x \in [W_1, W_2](H^2(\mathbb{D}, \mathbb{C}^2) \ominus \Phi H^2(\mathbb{D}, \mathbb{C}^{2})) = [W_1, W_2]\ker T_{\Phi^*}$, then $x = [W_1, W_2] \Phi \begin{pmatrix}
\overline{zp_1} \\
\overline{zp_2}
\end{pmatrix}  = [W_1, W_2] \left( C_1(\Phi\overline{z}) 
\overline{p_1} + C_2(\Phi\overline{z}) 
\overline{p_2} \right)$, for some $p_1, p_2 \in H^2$. If $[W_1,W_2]C_1(\Phi\overline{z}), [W_1,W_2]C_2(\Phi\overline{z}) \in \ker T_H$ for any bounded symbol $H$ then
$$
H [W_1,W_2]C_1(\Phi\overline{z}) \overline{p_1} \in \overline{N^+(\mathbb{D}, \mathbb{C}^2)} \text{ and } H [W_1,W_2]C_1(\Phi\overline{z}) \overline{p_2} \in \overline{N^+(\mathbb{D}, \mathbb{C}^2)}.
$$
Which means
$$
Hx = H  [W_1, W_2] \left( C_1(\Phi\overline{z}) 
\overline{p_1} + C_2(\Phi\overline{z}) 
\overline{p_2} \right)  \in \overline{N^+(\mathbb{D}, \mathbb{C}^2)},
$$
but as $x \in H^2(\mathbb{D}, \mathbb{C}^n)$ and $H$ is bounded we can further conclude $Hx \in \overline{H^2_0(\mathbb{D}, \mathbb{C}^2)}, $ and so $x \in \ker T_H$. Our $x \in \ker T_G$ was arbitrarily chosen so
$$
\ker T_G \subseteq \ker T_H.
$$
Thus $\ker T_G$ has a maximal pair of functions given by 
$$\{[W_1,W_2]C_1(\Phi \overline{z}), [W_1,W_2]C_2(\Phi \overline{z})\}$$.
\end{proof}
\begin{rem}
This result can be extended to show that if $$\ker T_G = [ W_1, W_2, ... W_n] (H^2(\mathbb{D}, \mathbb{C}^n) \ominus \Phi H^2(\mathbb{D}, \mathbb{C}^{n})),$$ then $[ W_1, W_2, ... W_n] C_i ( \Phi \overline{z})$ for $i \in \{ 1 \hdots n \}$ is a maximal n-tuple of  functions for $\ker T_G$.
\end{rem}

\section{Minimal kernel of multiple elements in $H^p$}

As previously mentioned it has been shown in \cite{MR3286053} that every $f \in H^p$ lies in a non-trivial Toeplitz kernel. If we try to consider the minimal kernel of two elements $f,g \in H^p$ we often find that $\kappa_{min}(f,g) = H^p$, furthermore this seems to have a connection to cyclic vectors for the backward shift. This is demonstrated with the following example.
\begin{exmp}
Let $f$ be a cyclic vector for the backward shift on $H^p$, then $\kappa_{min}(f, 1)$ is equal to $H^p$.

If for any symbol $h$, $f, 1 \in \ker T_h$ , then $f - \lambda \in \ker T_h$ for any $\lambda \in \mathbb{C}$. Hence $f - f(0) \in \ker T_h$, and by near invariance of Toeplitz kernels $\frac{f-f(0)}{z} = B(f) \in \ker T_h$. We can repeat this process inductively to give $B^n(f) \in \ker T_h$, for all $n \in \mathbb{N}$ and as $f$ is cyclic, we deduce $H^p \subseteq \ker T_h$.
\end{exmp}

The following theorem gives a sufficient condition for a given function $g$ to be the symbol of a Toeplitz operator whose kernel is the minimal kernel of a given set of functions in $H^p$. This result may be viewed as a partial generalisation of Theorem 2.2 in \cite{MR3806717}.

\begin{thm}\label{premain}
If $f_1 \hdots f_k \in H^p$ and $ g \in L^{\infty}$ are such that $gf_j = \overline{zp_j}$ for some $p_j \in H^p$ and $GCD(p_1^i \hdots p_k^i)=1$, then $\kappa_{min}(f_1 \hdots f_k) = \ker T_g$.
\end{thm}

\begin{proof}
It is clear that $f_j \in \ker T_g$ for all $j$. We can write $g$ as $g= \frac{\overline{zp_j}}{f_j}$, then for all $x \in \ker T_g$ we have have $xg=\overline{zp}$ for some $p \in H^p$. Substituting our expression for $g$ in to $xg=\overline{zp}$ we may write $\frac{x\overline{zp_j}}{f_j}= \overline{zp}$, and so $x= \frac{f_j p_j^i \overline{p}}{\overline{p_j^o}}$ and then $hx=\frac{p_j^i(hf_j) \overline{p}}{\overline{p_j^o}} \in L^p$. Therefore if $f_j \in \ker T_h$, by Proposition \ref{quotientouter} $\frac{(hf_j) \overline{p}}{\overline{p_j^o}} \in \overline{zN^+} \cap L^p = \overline{H^p_0}$, which means $hx=\frac{p_j^i(hf_j) \overline{p}}{\overline{p_j^o}} \in p_j^i\overline{H^p_0}$, so $P(hx) \in K_{p_j^i}$ for all $j$. As $GCD(p_1^i \hdots p_n^i)=1$ this means $P(hx) \in \bigcap_{j} K_{p_j^i} = K_{1} = \{ 0 \}$. We conclude $x \in \ker T_h$ and then $\ker T_g \subseteq \ker T_h$.
\end{proof}
Although the following corollary can also be obtained from Corollary 5.1 in \cite{MR3286053}, we give an alternate proof.
\begin{cor}\label{max existence}
Every non-trivial scalar Toeplitz kernel has a maximal function.
\end{cor}
\begin{proof}
Specialising the above theorem to $k=1$, we see that if there exists an $f \in H^p$ such that $gf = \overline{zp}$ where $p \in H^p$ is outer then $\kappa_{min}(f) = \ker T_g$. If $\ker T_g$ is non-trivial then there exists a $f'$ such that $g f' = \overline{z p'}$ for some $p' \in H^p$, multiplying both sides of this equality by $(p')^i$ we see that $f' (p')^i$ is a maximal function.
\end{proof}
\begin{rem}
Using the above corollary, we also obtain an explicit expression for a maximal function in a non-trivial Toeplitz kernel (when the symbol for the Toeplitz operator is known). This expression can also be derived from Theorem 2.2 in \cite{MR3806717}.
\end{rem}
The following corollary can also be proved as a consequence of Theorem 2.2 in \cite{MR3806717}, but again we write a proof here.
\begin{cor}\label{max function decomp}
If $m$ is a maximal function for $\ker T_g$ then $\ker T_g = m \overline{N^+} \cap H^p$.
\end{cor}
\begin{proof}
We first show the $\supseteq$ inclusion. As $m \in \ker T_g$, we must have $mg \overline{p} \in \overline{z N^+}$ for all $p \in N^+$, so consequently if $m \overline{p} \in H^p$ we would have $gm \overline{p} \in \overline{H^p_0}$. To show the $\subseteq$ inclusion we note that $gm = \overline{zp_1^o}$ where $p_1^o$ is an outer function in $H^p$, and if $f \in \ker T_g$ then $gf = \overline{zp_2}$ where $p_2 \in H^p$. Solving these expressions for $f$ we see $f= m \frac{\overline{zp_2}}{\overline{p_1^o}} \in m \overline{N^+} \cap H^p$.
\end{proof}
For clarity in the following theorem we will write $span^{N^+}$ to mean the closed linear span in $N^+$, and we will write $span$ to mean the linear span.

\begin{thm}\label{min}
Let $f,g \in H^p$. If $ \frac{g}{f^o}$ is cyclic for the backward shift on $N^+$ then $\kappa_{min}(f, g) = H^p$.
\end{thm}
\begin{proof}
For any bounded $h$, if $f, g \in \ker T_h$ then near invariance shows $f^o \in \ker T_h$, and so for any $\lambda \in \mathbb{C}$, $$g - \lambda f^o = f^o(\frac{g}{f^o} - \lambda) \in \ker T_h.$$ Letting $\lambda = \frac{g}{f^o}(0)$ we see that $$f^o(\frac{g}{f^o} - \frac{g}{f^o}(0)) \in \ker T_h,$$ and near invariance gives $$f^o\frac{(\frac{g}{f^o} - \frac{g}{f^o}(0))}{z} = f^oB(\frac{g}{f^o}) \in \ker T_h. $$
We can repeat this process inductively to give \begin{equation}\label{closure}
  span\{ f^oB^n(\frac{g}{f^o}) \} \subseteq \ker T_h. 
\end{equation}

We now take the closure of both sides of this set inclusion in the $H^p$ subspace topology of $N^+$. We first show $span^{N^+}\{ f^oB^n(\frac{g}{f^o}) \} = N^+$.

We have $f^o \in N^+$ and $B^n(\frac{g}{f^o}) \in N^+$, so as $N^+$ is closed under multiplication $\{ f^oB^n(\frac{g}{f^o}) \} \subseteq N^+$ and hence $span^{N^+}\{ f^oB^n(\frac{g}{f^o}) \} \subseteq N^+$, so one set inclusion is clear. We now show $ N^+$ is contained in $span^{N^+}\{ f^oB^n(\frac{g}{f^o}) \} $. Take any $ x \in N^+$ then as $\frac{g}{f^o}$ is cyclic for $N^+$ and $\frac{x}{f^o} \in N^+$ there exists an $(x_k) \subseteq $span$\{ B^n(\frac{g}{f^o}) \}$ such that $x_k \to \frac{x}{f^o}$ in $N^+$. Then as $N^+$ is a topological algebra we must have $f^o x_k \to x$ in $N^+$. So the closure of the left hand side of \eqref{closure} in the $H^p$ subspace topology of $N^+$ is equal to $ N^+ \cap H^p = H^p$.

The  closure of the right hand side of \eqref{closure} in the $H^p$ subspace topology of $N^+$ is the closure of $\ker T_h$ in $N^+$ intersected with $H^p$. This can be seen to equal $\ker T_h$ via the following observation. Let $
x_k \in \ker T_h \subseteq N^+$ be such that $x_k \to x$ in $\log L$ (or equivalently $N^+$), then as $\log L$ is a topological algebra $\overline{z h x_k} \to \overline{z h x}$ in $\log L$. As $ \overline{z h x_k} \in N^+ $ and $N^+$ is closed in $\log L$ so we must have $ \overline{z h x} \in N^+ $. If $x \in H^p $ then $\overline{z h x} \in N^+ \cap L^p = H^p$ so $x \in \ker T_h$. We conclude
$$
H^p \subseteq \ker T_h.
$$
\end{proof}

\begin{cor}
Let $f_1 \hdots f_k \in H^p$. If for any pair $f_j, f_l$ with $j,l = 1 \hdots k$, $\frac{f_j}{f_l^o}$ is a cyclic vector for the backward shift on $N^+$, then $\kappa_{min}(f_1 \hdots f_k) = H^p$.
\end{cor}
We now find a minimal kernel for when $\frac{g}{f^o}$ is not a cyclic vector for the backward shift. It is immediate that if $\frac{g}{f^o}$ is not cyclic for $N^+$ then it lies inside some $B$ invariant subspace, and so to further understand $\kappa_{min}(f,g) $ we must discuss the $B$ invariant subspaces of $N^+$. As far as the author is aware the $B$ invariant subspaces of $N^+$ have not been described, however the following (unproved) conjecture is due to Aleksandrov and can be found in section 11.15 of \cite{havin2006linear}.
\begin{conj}
\end{conj}

The $B$ invariant subspaces of $N^+$ depend on three parameters:
\begin{enumerate}
 \item{An inner function $\theta$.}
 \item{A closed set $F \subseteq \mathbb{T}$ with $\sigma(\theta) \cap \mathbb{T} \subseteq F$, where $$\sigma(\theta)= \{ z \in \mathbb{D}^- : \lim \inf_{\lambda \xrightarrow{} z} |\theta(\lambda)| =0 \}$$ is the spectrum of an inner function $\theta$.}
 \item{A function $k: F \to \mathbb{N} \cup \{\infty\}$ with the additional property
 $
k( \eta ) = \infty
 $
 for all $\eta \in \sigma( \theta ) \cap \mathbb{T}$ and for all non-isolated points $\eta \in F$.}
\end{enumerate}

Define $\mathcal{E}(\theta, F, k)$ to be the set of $f \in N^+$ with:
\begin{enumerate}
    \item{  $\overline{z}\theta \overline{f} \in N^+$.}
    \item{$f$ has an meromorphic continuation $\tilde{f}$ to a neighbourhood of $\hat{\mathbb{C}} \setminus F$.}
    \item{$\eta$ is a pole of $\tilde{f}$ of order at most $k(\eta)$ for all $\eta \in F$ with $k(\eta) \neq \infty$. }
\end{enumerate}
Then $\mathcal{E}(\theta, F, k)$ is a proper $B$ invariant subspace of $N^+$ and for every non-trivial $B$ invariant subspace $\mathcal{E} \subset N^+$, there is a triple $(\theta, F, k)$ such that $\mathcal{E} =\mathcal{E}(\theta, F, k) $.
 \vskip 1cm
 
We will focus on $B$ invariant subspaces of $N^+$ of the form $\{ f \in N^+ : \overline{z} \theta \overline{f} \in N^+ \} =:\theta^* (N^+)$, where $\theta$ is some fixed inner function and the above multiplication is understood on $\mathbb{T}$. We call $B$ invariant subspaces of this form \textit{one component} $B$ invariant subspaces.

\begin{prop}\label{p2}
Let $\tau$ be a family of inner functions, then $$ \bigcap_{\theta \in \tau} \theta^*(N^+) = GCD( \tau  )^*(N^+). $$
\end{prop}
\begin{proof}
The $\supseteq $ is clear. To prove the $\subseteq$ inclusion we start with the fact that the $H^2$ closure of $span \{ \theta H^2 : \theta \in \tau \}$ is equal to $GCD( \tau) H^2$. This means we can find a sequence $ h_n \in span\{ \theta H^2 : \theta \in \tau \}$ such that $ h_n \to GCD( \tau) $ in the $H^2$ norm (which also then implies convergence in $N^+$). So if $f \in \cap_{\theta \in \tau} \theta^*(N^+) $, then $\overline{z} \theta \overline{f} \in N^+$ for all $\theta \in \tau$, in particular as $N^+$ is an algebra $\overline{z} h_n \overline{f} \in N^+$. Taking the limit in the metric of $\log L$, noting $\log L$ is a topological algebra and $N^+$ is closed we see that $\overline{z} GCD( \tau ) \overline{f} \in N^+ $. 
\end{proof}
Although the $B$ invariant subspaces of $N^+$ have not been completely described, there is a partial result showing all $B$ invariant subspaces of $N^+$ are contained in a one component $B$ invariant subspace. The following can be found as Corollary 1, page 42 in \cite{havin2006linear}.
\begin{prop}\label{bigger space}
Given a non-trivial $B$ invariant subspace of $N^+$, $\mathcal{E}$, there exists an inner function $I$ such that $\mathcal{E} \subseteq I^*(N^+).$
\end{prop}

If $\frac{g}{f^o}$ is not cyclic, from the above proposition there exists a $\theta$ such that $\frac{g}{f^o}$ lies in $\theta^* (N^+)$. It then follows $f^i, \frac{g}{f^o}$ lie in a one component $B$ invariant subspace ($(\theta f^i) ^*(N^+)$ is one such example). Then Theorem \ref{p2} allows us to talk about the smallest one component $B$ invariant subspace containing $f^i, \frac{g}{f^o}$.

\begin{thm}\label{3.1}
Let $f, g \in H^p$. If $\frac{g}{f^o}$ is not cyclic for $B$ then $\kappa_{min} (f,g) = \ker T_{\overline{f^o}{\overline{\theta}}/ f^o}$, where $\theta$ is such that $\theta^*(N^+)$ is the smallest one component $B$ invariant subspace containing both $\frac{g}{f^o}$ and $f^i$.
\end{thm}
\begin{proof}
We first show $f, g \in \ker T_{\overline{f^o}{\overline{\theta}}/ f^o}$. As $\frac{g}{f^o}, f^i  \in \theta^*(N^+)$, 
$$
\frac{g}{f^o} = \theta \overline{zp_1},
$$
and
$$
f^i = \theta \overline{zp_2},
$$
for some $p_1, p_2 \in N^+$. So 
$$
g(\frac{\overline{f^o}{\overline{\theta}}}{ f^o} ) = \overline{f^o z p_1},$$
and
$$
f (\frac{\overline{f^o}{\overline{\theta}}}{ f^o} ) = \overline{f^o z \ p_2},
$$ both of which are in $\overline{zN^+} \cap L^p = \overline{H^p_0}$ (both can be seen to lie in $L^P$ because the symbol for the operator is unimodular). Now by Theorem \ref{premain} all that remains to be proved is that $GCD(p_1^i, p_2^i) = 1$.

Because $f^i$ is inner this then forces $p_2$ to be inner. If $GCD(p_2, p_1^i) = \alpha \neq 1$ then as $p_2 | \theta$, this then forces $\alpha| \theta $ and then this would imply $\frac{g}{f^o}, f^i \in (\theta\overline{\alpha})^*(N^+) \subseteq \theta(N^+)$. Which can not be the case by minimality of our choice of $\theta$.
\end{proof}
Combining Theorem \ref{min} and Theorem \ref{3.1} we can now give a complete answer as to when $\kappa_{min}(f,g) = H^p$. This characterisation allows us to deduce an equivalent condition for a function to be cyclic for the backward shift on $N^+$.

\begin{cor}\label{1cor}
Let $f, g \in H^p$. There are no non-trivial Toeplitz kernels containing both $f$ and $g$ if and only if $\frac{g}{f^o}$ is a cyclic vector for the backward shift on $N^+$.
\end{cor}
Due to symmetry of the above corollary and the fact that every outer function in $N^+$ can be expressed as the quotient of two outer functions in $H^p$ we can also deduce the following.
\begin{cor}
Let $f \in N^+$ and outer, then $f$ is cyclic for the backward shift if and only if $\frac{1}{f}$ is a cyclic vector for the backward shift.
\end{cor}

\section{Minimal kernel of multiple elements in $H^p(\mathbb{D}, \mathbb{C}^2)$}
Keeping with earlier notation we will use Greek symbols for elements of $H^p(\mathbb{D}, \mathbb{C}^2)$. When considering the minimal kernel of $\begin{pmatrix}
\phi_1 \\
\phi_2
\end{pmatrix},
\begin{pmatrix}
\psi_1 \\
\psi_2
\end{pmatrix} \in H^p(\mathbb{D}, \mathbb{C}^2 )$, we find that the minimal kernel depends on the determinant of $ M= \begin{pmatrix}
\phi_1  & \psi_1  \\
\phi_2  & \psi_2  \\
\end{pmatrix} $. We first consider the case when  $\det{M} = \phi_1 \psi_2 - \psi_1 \phi_2$ is not identically equal to zero.
\begin{thm}
Let $\begin{pmatrix}
\phi_1 \\
\phi_2
\end{pmatrix},
\begin{pmatrix}
\psi_1 \\
\psi_2
\end{pmatrix} \in H^p(\mathbb{D}, \mathbb{C}^2 )$. If $\phi_1 \psi_2 - \psi_1 \phi_2$ is not identically equal to zero then $\kappa_{min} (\begin{pmatrix}
\phi_1 \\
\phi_2
\end{pmatrix},
\begin{pmatrix}
\psi_1 \\
\psi_2
\end{pmatrix})  = \ker T_{(\overline{u_1}/\overline{u_2}) \overline{z} M^{-1}},$ where $u_1$ is a scalar outer function with $|u_1| = |\phi_1 \psi_2 - \psi_1 \phi_2| $, and $u_2$ is a scalar outer function with $|u_2| = |\phi_1| + |\phi_2| + |\psi_1| + |\psi_2| + 1$.
\end{thm}
\begin{proof}
We first note that the specified symbol is in fact bounded. We have $$(\overline{u_1}/\overline{u_2})\overline{z} M^{-1} = \overline{z} \frac{\overline{u_1}}{\phi_1 \psi_2 - \psi_1 \phi_2} \begin{pmatrix}
\psi_2/\overline{u_2} & -\psi_1/\overline{u_2} \\
-\phi_2/\overline{u_2} & \phi_1/\overline{u_2} \\
\end{pmatrix},
$$by construction $|\overline{z} \frac{\overline{u_1}}{\phi_1 \psi_2 - \psi_1 \phi_2}| = 1$ and each entry in $\begin{pmatrix}
\psi_2/\overline{u_2} & -\psi_1/\overline{u_2} \\
-\phi_2/\overline{u_2} & \phi_1/\overline{u_2} \\
\end{pmatrix}$ has modulus smaller than 1, hence $(\overline{u_1}/\overline{u_2})\overline{z} M^{-1}$ is a bounded matrix symbol.

For any 
$\begin{pmatrix}
f_1 \\
f_2
\end{pmatrix} \in \ker T_{(\overline{u_1}/\overline{u_2}) \overline{z} M^{-1}} $, we have 
$$
(\overline{u_1}/\overline{u_2}) \overline{z} M^{-1} 
\begin{pmatrix}
f_1 \\
f_2
\end{pmatrix} \in \overline{H^p_0(\mathbb{D}, \mathbb{C}^2 )}.
$$
Dividing through by  $\overline{u_1}$,  then multiplying through by $\overline{u_2}$ we see that
$$
\overline{z} M^{-1} 
\begin{pmatrix}
f_1 \\
f_2
\end{pmatrix} = \begin{pmatrix}
\overline{zp_1} \\
\overline{zp_2}
\end{pmatrix},
$$
for some $p_1, p_2 \in N^+$, so$$
\begin{pmatrix}
f_1 \\
f_2
\end{pmatrix} = M \begin{pmatrix}
 \overline{p_1} \\
 \overline{p_2}
\end{pmatrix} = \begin{pmatrix}
\phi_1 \\
\phi_2
\end{pmatrix} \overline{p_1} + \begin{pmatrix}
\psi_1 \\
\psi_2
\end{pmatrix} \overline{p_2}.
$$
Then for any other bounded matrix $H$ we have
$$
H \begin{pmatrix}
f_1 \\
f_2
\end{pmatrix} = H \begin{pmatrix}
\phi_1 \\
\phi_2
\end{pmatrix} \overline{p_1} + H\begin{pmatrix}
\psi_1 \\
\psi_2
\end{pmatrix} \overline{p_2}.
$$
So if $\begin{pmatrix}
\phi_1 \\
\phi_2
\end{pmatrix},
\begin{pmatrix}
\psi_1 \\
\psi_2
\end{pmatrix} \in \ker T_{H}$, then both coordinates of $H \begin{pmatrix}
f_1 \\
f_2
\end{pmatrix}$ lie in $L^p$ and both $H \begin{pmatrix}
\phi_1 \\
\phi_2
\end{pmatrix} \overline{p_1}$ and $H\begin{pmatrix}
\psi_1 \\
\psi_2
\end{pmatrix} \overline{p_2}$ have both their co-ordinates lying in $\overline{zN^+}$, so therefore $H \begin{pmatrix}
f_1 \\
f_2
\end{pmatrix} \in \overline{H^p_0(\mathbb{D}, \mathbb{C}^2 )} $. We conclude 
$$
 \ker T_{(\overline{u_1}/\overline{u_2}) \overline{z} M^{-1}} \subseteq \ker T_H.
$$
\end{proof}
We now consider the minimal kernel for when $\phi_1\psi_2 - \psi_1\phi_2 = 0$. In the following we let $P_1$ and $P_2$ denote the projections $L^p(\mathbb{D}, \mathbb{C}^2) \to H^p$ on to the first and second coordinate respectively. 

\begin{thm}
Let $\begin{pmatrix}
\phi_1 \\
\phi_2
\end{pmatrix},
\begin{pmatrix}
\psi_1 \\
\psi_2
\end{pmatrix} \in H^p(\mathbb{D}, \mathbb{C}^2 )$ and let $u$ be an outer function such that $|u| = |\phi_1| + |\phi_2| + 1$. If $\frac{\psi_2}{\phi_2^o}$ is not a cyclic vector for the backward shift on $N^+$ and $\phi_1\psi_2 - \psi_1\phi_2 = 0$, then we have 
$$
\kappa_{min}(\begin{pmatrix}
\phi_1 \\
\phi_2
\end{pmatrix},
\begin{pmatrix}
\psi_1 \\
\psi_2
\end{pmatrix}) = \ker T_{ \begin{pmatrix}
\phi_2/u & -\phi_1/u \\
0 & \overline{\phi_2^o}\overline{\theta}/\phi_2^o \\
\end{pmatrix} },
$$
where $\theta$ is such that $\theta^*(N^+)$ is the smallest one component $B$ invariant subspace containing both $\frac{\psi_2 }{\phi_2^o}$ and $\phi_2^i$.

\begin{rem}
We note how $\theta$ is the same inner function that appears in the symbol for the scalar minimal kernel of $\phi_2$ and $\psi_2$.
\end{rem}

\end{thm}
\begin{proof}
Our choice of $\theta$ guarantees both the vectors are in the required kernel.

Let $ \begin{pmatrix}
x_1 \\
x_2
\end{pmatrix} \in \ker T_{ \begin{pmatrix}
\phi_2/u & -\phi_1/u \\
0 & \overline{\phi_2^o}\overline{\theta}/\phi_2^o \\
\end{pmatrix}}$, then we have 
$$
x_2 = \frac{\phi^o_2 \theta \overline{zp}}{\overline{\phi^o_2}},
$$
for some $p \in H^p$. As in the scalar case for our choice of $\theta$  we have $\frac{\psi_2}{\phi_2^o} = \theta \overline{zp_1}$ and $\phi^i = \theta \overline{zp_2}$, for some $p_1, p_2 \in N^+$, so $\theta$ can be written as $$
\theta= \frac{\psi_2 z p_1^i}{\phi^o_2 \overline{p_1^o}},
$$
and
$$
\theta = \phi^i_2 z p_2,$$ where $p_2$ is inner. Substituting our two expressions for $\theta$ in to the above expression for $x_2$ gives
\begin{equation}\label{1}
   x_2 = \frac{\psi_2 p_1^i \overline{p}}{\phi_2^o \overline{p_1^o}}, 
\end{equation}
and
\begin{equation}\label{2}
    x_2 = \frac{\phi_2 p_2 \overline{p}}{\overline{\phi^o_2}}.
\end{equation}
We also have that $\begin{pmatrix}
x_1 \\
x_2
\end{pmatrix} $ satisfies
$$
x_1 \phi_2 - \phi_1 x_2 = 0,
$$
so substituting $x_2 = \frac{\phi_2 p_2 \overline{p}}{\overline{\phi^o_2}}$ from \eqref{2}
yields
$$
x_1 \phi_2 - \phi_1 \frac{\phi_2 p_2 \overline{p}}{\overline{\phi^o_2}} = 0,
$$
and so
$$
x_1 = \phi_1 \frac{ p_2 \overline{p}}{\overline{\phi^o_2}}.
$$
Consequently we may write all $ \begin{pmatrix}
x_1 \\
x_2
\end{pmatrix} \in \ker T_{ \begin{pmatrix}
\phi_2/u & -\phi_1/u \\
0 & \overline{\phi_2^o}\overline{\theta}/\phi_2^o \\
\end{pmatrix}}$ are of the form
$$
 \begin{pmatrix}
x_1 \\
x_2
\end{pmatrix} = \begin{pmatrix}
\phi_1 \\
\phi_2
\end{pmatrix} \frac{p_2 \overline{p}}{\overline{\phi_2^o}}.
$$
We will now find a similar expression relating  $\begin{pmatrix}
x_1 \\
x_2
\end{pmatrix}$ and $ \begin{pmatrix}
\psi_1 \\
\psi_2
\end{pmatrix}$. Multiplying 
$$
x_1 \phi_2 - \phi_1 x_2 = 0,
$$ by $\frac{\psi_1}{\phi_1} = \frac{\psi_2}{\phi_2}$ gives
$$
x_1 \psi_2 - \psi_1 x_2 = 0,
$$
and substituting $x_2 = \frac{\psi_2  p_1^i \overline{p}}{\overline{\phi^o_2} \overline{p_1^o}}$ from \eqref{1} in to this expression yields
$$
x_1 \psi_2 - \psi_1 \frac{\psi_2  p_1^i \overline{p}}{\overline{\phi^o_2} \overline{p_1^o}} = 0,
$$
so
$$
x_1 = \psi_1 \frac{ p_1^i \overline{p}}{\overline{\phi^o_2} \overline{p_1^o}}.
$$
Consequently we can write
$$
\begin{pmatrix}
x_1 \\
x_2
\end{pmatrix} = \begin{pmatrix}
\psi_1 \\
\psi_2
\end{pmatrix} \frac{ p_1^i \overline{p}}{\overline{\phi^o_2} \overline{p_1^o}}.
$$
Now we have two expressions for $\begin{pmatrix}
x_1 \\
x_2
\end{pmatrix} \in \ker T_{ \begin{pmatrix}
\phi_2/u & -\phi_1/u \\
0 & \overline{\phi_2^o}\overline{\theta}/\phi_2^o \\
\end{pmatrix}}$,
$$
\begin{pmatrix}
x_1 \\
x_2
\end{pmatrix} = \begin{pmatrix}
\psi_1 \\
\psi_2
\end{pmatrix} \frac{ p_1^i \overline{p}}{\overline{\phi^o_2} \overline{p_1^o}},
$$
and
$$
 \begin{pmatrix}
x_1 \\
x_2
\end{pmatrix} = \begin{pmatrix}
\phi_1 \\
\phi_2
\end{pmatrix} \frac{p_2 \overline{p}}{\overline{\phi_2^o}}.
$$
So if $\begin{pmatrix}
\phi_1 \\
\phi_2
\end{pmatrix},
\begin{pmatrix}
\psi_1 \\
\psi_2
\end{pmatrix} \in \ker T_H$, for any symbol $H$, then
$$
H \begin{pmatrix}
x_1 \\
x_2
\end{pmatrix}  = H \begin{pmatrix}
\psi_1 \\
\psi_2
\end{pmatrix} \frac{ p_1^i \overline{p}}{\overline{\phi^o_2} \overline{p_1^o}} = \Big(  H \begin{pmatrix}
\psi_1 \\
\psi_2
\end{pmatrix} \Big) \Big( \frac{ p_1^i \overline{p}}{\overline{\phi^o_2} \overline{p_1^o}} \Big).
$$
By Proposition \ref{quotientouter} $\frac{  \overline{p}}{\overline{\phi_2^o o_1^o} } \in \overline{N^+}$ and $H\begin{pmatrix}
\psi_1 \\
\psi_2
\end{pmatrix} \in \overline{H^p_0(\mathbb{D}, \mathbb{C}^2)}$, so both coordinates of $\Big(  H \begin{pmatrix}
\psi_1 \\
\psi_2
\end{pmatrix} \Big) \Big( \frac{ \overline{p}}{\overline{\phi^o_2} \overline{p_1^o}} \Big) 
$ are in $\overline{zN^+} \cap L^p(\mathbb{T}) = \overline{H^p_0}$, and so $ H \begin{pmatrix}
x_1 \\
x_2
\end{pmatrix} =  \Big(  H \begin{pmatrix}
\psi_1 \\
\psi_2
\end{pmatrix} \Big) \Big( \frac{ p_1^i \overline{p}}{\overline{\phi^o_2} \overline{p_1^o}} \Big) \in  p_1^i \overline{H^2_0(\mathbb{D}, \mathbb{C}^2 )}$. Similarly 
$$
H\begin{pmatrix}
x_1 \\
x_2
\end{pmatrix}  = H\begin{pmatrix}
\phi_1 \\
\phi_2
\end{pmatrix} \frac{p_2 \overline{p}}{\overline{\phi_2^o}} =  \Big(  H \begin{pmatrix}
\phi_1 \\
\phi_2
\end{pmatrix} \Big) \Big(  \frac{p_2 \overline{p}}{\overline{\phi_2^o}} \Big) \in p_2 \overline{H^p_0(\mathbb{D}, \mathbb{C}^2 )}.
$$
So $P_1(H\begin{pmatrix}
x_1 \\
x_2
\end{pmatrix}) \in K_{p_2} \cap K_{p_1^i} = K_{GCD(p_2, p_1^i)} $, but as in the scalar case we have chosen $\theta$ such that $GCD(p_2, p_1^i) = 1$, so $P_1(H\begin{pmatrix}
x_1 \\
x_2
\end{pmatrix}) \in K_1 = \{ 0 \} $. The same holds for $P_2(H\begin{pmatrix}
x_1 \\
x_2
\end{pmatrix}) $ and so $P(H\begin{pmatrix}
x_1 \\
x_2
\end{pmatrix}) = 0 $, and therefore
$$
\ker T_{ \begin{pmatrix}
\phi_2/u & -\phi_1/u \\
0 & \overline{\phi_2^o}\overline{\theta}/\phi_2^o \\
\end{pmatrix}} \subseteq \ker T_H.
$$
\end{proof}
\vskip 0.5cm
We now consider the case when $\frac{\psi_2}{\phi_2^o}$ is cyclic for $B$. In doing so we need to introduce some new theory. Let $N^+(\mathbb{D},\mathbb{C}^2) := \{ \begin{pmatrix}
f_1 \\
f_2
\end{pmatrix} : f_1, f_2 \in N^+ \}$ with the metric on $N^+(\mathbb{D},\mathbb{C}^2)$ defined by 
$$
\rho^2\left(  \begin{pmatrix}
f_1 \\
f_2
\end{pmatrix}, \begin{pmatrix}
g_1 \\
g_2
\end{pmatrix} \right) = \rho(f_1,g_1) + \rho(f_2,g_2),
$$
where $\rho$ is the metric on $N^+$. It is easily checked that $N^+(\mathbb{D},\mathbb{C}^2)$ is also a metric space and a sequence in $N^+(\mathbb{D},\mathbb{C}^2)$ converges if and only if both of its coordinates converge in $N^+$. As outer functions are invertible in $N^+$ for a fixed $f \in N^+$, $f N^+ = f^i N^+$ is closed. For a fixed $\begin{pmatrix}
f_1 \\
f_2
\end{pmatrix} \in N^+(\mathbb{D}, \mathbb{C}^2)$, the following computation shows $\begin{pmatrix}
f_1 \\
f_2
 \end{pmatrix} N^+$ is closed in $N^+(\mathbb{D},\mathbb{C}^2)$. If $\begin{pmatrix}
f_1 \\
f_2
 \end{pmatrix}x_n \to \begin{pmatrix}
x_1 \\
x_2
 \end{pmatrix}$ then $f_1 x_n \to x_1$ so $x_1 = f_1 x_0$, for some $x_0 \in N^+$, then as $\log L$ is a topological algebra we can deduce $x_n \to x_0$. So then $f_2 x_n \to f_2 x_0$ and $\begin{pmatrix}
f_1 \\
f_2
 \end{pmatrix}x_n \to \begin{pmatrix}
f_1 \\
f_2
 \end{pmatrix} x_0 \in \begin{pmatrix}
f_1 \\
f_2
 \end{pmatrix} N^+$.

We can also let $\rho^2$ define a metric on $\log L (\mathbb{D}, \mathbb{C}^2) = \{ \begin{pmatrix}
f_1 \\
f_2
\end{pmatrix} : f_1, f_2 \in \log L \}$ and in this metric $N^+(\mathbb{D},\mathbb{C}^2)$ is a closed subspace of $\log L (\mathbb{D}, \mathbb{C}^2)$.

\begin{thm}
Let $\begin{pmatrix}
\phi_1 \\
\phi_2
\end{pmatrix},
\begin{pmatrix}
\psi_1 \\
\psi_2
\end{pmatrix} \in H^p(\mathbb{D}, \mathbb{C}^2 )$, let $\beta = GCD( \phi_1^i, \phi_2^i)$ and let $u$ be an outer function such that $|u| = |\phi_1| + |\phi_2| + 1$. If  $\frac{\psi_2}{\overline{\beta}\phi_2}$ is a cyclic vector for the backward shift on $N^+$ and $\phi_1\psi_2 - \psi_1\phi_2 = 0$, then we have

$$
\kappa_{min} \left( \begin{pmatrix}
\phi_1 \\
\phi_2
\end{pmatrix},
\begin{pmatrix}
\psi_1 \\
\psi_2
\end{pmatrix} \right) = \ker T_{ \begin{pmatrix}
\phi_2/u & -\phi_1/u \\
0 & 0 \\
\end{pmatrix} }.
$$
\end{thm}

The assumption $\phi_1\psi_2 - \psi_1\phi_2 = 0$ ensures $\frac{\psi_2}{\overline{\beta}\phi_2} \in N^+$, as $\overline{\beta}\phi_1\psi_2 = \psi_1\overline{\beta}\phi_2$ and $GCD(\overline{\beta} \phi_1^i, \overline{\beta}\phi_2^i) = 1$ so every inner factor of $\overline{\beta}\phi_2$ divides $\psi_2$.

In the following proof we will write $span^{N^+}$ to mean the closed linear span in $N^+(\mathbb{D}, \mathbb{C}^2)$, and $span$ to mean the linear span.
\begin{proof}
We split the proof up in to two stages. We first prove if for any bounded symbol $H$ we have $\begin{pmatrix}
\phi_1 \\
\phi_2
\end{pmatrix},
\begin{pmatrix}
\psi_1 \\
\psi_2
\end{pmatrix} \in \ker T_H $, then $ \overline{\beta}\begin{pmatrix}
\phi_1 \\
\phi_2
\end{pmatrix}N^{+} \cap H^p(\mathbb{D}, \mathbb{C}^2 ) \subseteq \ker T_H$. Then we prove $\ker T_{ \begin{pmatrix}
\phi_2/u & -\phi_1/u \\
0 & 0 \\
\end{pmatrix} } = \overline{\beta}\begin{pmatrix}
\phi_1 \\
\phi_2
\end{pmatrix}N^{+} \cap H^p(\mathbb{D}, \mathbb{C}^2 ) $.

If $\begin{pmatrix}
\phi_1 \\
\phi_2
\end{pmatrix},
\begin{pmatrix}
\psi_1 \\
\psi_2
\end{pmatrix} \in \ker T_H $ then near invariance of Toeplitz kernels guarantees $\overline{\beta}\begin{pmatrix}
\phi_1 \\
\phi_2
\end{pmatrix} \in \ker T_H$, and so for $\lambda \in \mathbb{C}$
$$
\begin{pmatrix}
\psi_1 \\
\psi_2
\end{pmatrix} - \lambda \overline{\beta}\begin{pmatrix}
\phi_1 \\
\phi_2
\end{pmatrix} = \overline{\beta}\begin{pmatrix}
\phi_1 (\frac{\psi_1}{ \overline{\beta} \phi_1} - \lambda)  \\
\phi_2 (\frac{\psi_2}{\overline{\beta} \phi_2} - \lambda)
\end{pmatrix}
 \in \ker T_H.
$$
Noting $\frac{\psi_1}{ \overline{\beta} \phi_1} = \frac{\psi_2}{\overline{\beta} \phi_2}$, and letting $\lambda = \frac{\psi_1}{ \overline{\beta} \phi_1}(0) = \frac{\psi_2}{\overline{\beta} \phi_2}(0)  $ we see that,
$$
\overline{\beta}\begin{pmatrix}
\phi_1 (\frac{\psi_2}{ \overline{\beta} \phi_2} -\frac{\psi_2}{\overline{\beta} \phi_2}(0))  \\
\phi_2 (\frac{\psi_2}{\overline{\beta} \phi_2} - \frac{\psi_2}{\overline{\beta} \phi_2}(0))
\end{pmatrix} \in \ker T_H,
$$
and near invariance of Toeplitz kernels gives
$$
\frac{\overline{\beta}\begin{pmatrix}
\phi_1 (\frac{\psi_2}{ \overline{\beta} \phi_2} -\frac{\psi_2}{\overline{\beta} \phi_2}(0))  \\
\phi_2 (\frac{\psi_2}{\overline{\beta} \phi_2} - \frac{\psi_2}{\overline{\beta} \phi_2}(0))
\end{pmatrix}}{z} = \overline{\beta}\begin{pmatrix}
\phi_1 \\
\phi_2
\end{pmatrix} B(\frac{\psi_2}{\overline{\beta} \phi_2})  \in \ker T_H.
$$
We can repeat this process inductively to get $\overline{\beta}\begin{pmatrix}
\phi_1 \\
\phi_2
\end{pmatrix} B^n(\frac{\psi_2}{\overline{\beta} \phi_2})  \in \ker T_H$, and hence
\begin{equation}\label{span}
    span\{ \overline{\beta}\begin{pmatrix}
\phi_1 \\
\phi_2
\end{pmatrix} B^n(\frac{\psi_2}{\overline{\beta} \phi_2}) : n \in \mathbb{N}_0 \} \subseteq \ker T_H.
\end{equation}

We will now take the closure of both sides of this set inclusion in the $H^p(\mathbb{D}, \mathbb{C}^2 )$ subspace topology of $N^+(\mathbb{D},\mathbb{C}^2)$.

The closure of the left hand side of \eqref{span} is equal to $span^{N^+}\{ \overline{\beta}\begin{pmatrix}
\phi_1 \\
\phi_2
\end{pmatrix} B^n(\frac{\psi_2}{\overline{\beta} \phi_2}) \}$  intersected with $H^p(\mathbb{D}, \mathbb{C}^2 )$. As $\overline{\beta}\begin{pmatrix}
\phi_1 \\
\phi_2
\end{pmatrix}N^{+}$ is closed, $\frac{\psi_2}{\overline{\beta} \phi_2}$ is cyclic and $N^+$ is a topological algebra the closure of the left hand side of \eqref{span} equals $\overline{\beta}\begin{pmatrix}
\phi_1 \\
\phi_2
\end{pmatrix}N^{+} \cap H^p(\mathbb{D}, \mathbb{C}^2 )$.

The closure of the right hand side of \eqref{span} is the closure of $\ker T_H$ in $N^+(\mathbb{D},\mathbb{C}^2)$ intersected with $H^p(\mathbb{D}, \mathbb{C}^2 )$. We now argue this is equal to $\ker T_H$. Let $
\begin{pmatrix}
x_{1 n} \\
x_{2  n}
\end{pmatrix} \in \ker T_H$ be such that $\begin{pmatrix} x_{1 n} \\ x_{2  n} \end{pmatrix} \rightarrow \begin{pmatrix} x_1 \\ x_2 \end{pmatrix}$ in $N^+(\mathbb{D},\mathbb{C}^2)$, then $\begin{pmatrix} x_{1 n} \\ x_{2  n} \end{pmatrix} \rightarrow \begin{pmatrix} x_1 \\ x_2 \end{pmatrix}$ in $\log L(\mathbb{D},\mathbb{C}^2)$. As $\log L$ is a topological algebra and $H\begin{pmatrix} x_{1 n} \\ x_{2  n} \end{pmatrix} = \begin{pmatrix}
h_{11} x_{1 n} + h_{12} x_{2  n} \\
h_{21} x_{2  n} + h_{22} x_{2  n}
\end{pmatrix}$, we must have $\overline{zH \begin{pmatrix}
x_{1 n} \\
x_{2  n}
\end{pmatrix}} \rightarrow \overline{zH\begin{pmatrix}
x_1 \\
x_2
\end{pmatrix}} $ in $\log L(\mathbb{D},\mathbb{C}^2)$. As $\begin{pmatrix}
x_{1 n} \\
x_{2  n}
\end{pmatrix} \in \ker T_H$ we have  $ \overline{zH \begin{pmatrix} x_{1 n} \\ x_{2  n} \end{pmatrix}}  \in N^+(\mathbb{D},\mathbb{C}^2) $, and as $N^+(\mathbb{D},\mathbb{C}^2) $ is closed in $\log L(\mathbb{D},\mathbb{C}^2)$ we must have $ \overline{zH \begin{pmatrix} x_1 \\ x_2 \end{pmatrix}}   \in N^+(\mathbb{D},\mathbb{C}^2) $. So if $\begin{pmatrix} x_1 \\ x_2 \end{pmatrix} \in H^p(\mathbb{D}, \mathbb{C}^2 ) $ then  $ \overline{zH \begin{pmatrix} x_1 \\ x_2 \end{pmatrix}}   \in N^+(\mathbb{D},\mathbb{C}^2) \cap L^p(\mathbb{D}, \mathbb{C}^2 ) = H^p(\mathbb{D}, \mathbb{C}^2 ) $, so $\begin{pmatrix} x_1 \\ x_2 \end{pmatrix} \in \ker T_H$. From this we deduce
$$
\overline{\beta}\begin{pmatrix}
\phi_1 \\
\phi_2
\end{pmatrix}N^{+} \cap H^p(\mathbb{D}, \mathbb{C}^2 ) \subseteq \ker T_H.
$$

It remains to prove that $\ker T_{ \begin{pmatrix}
\phi_2/u & -\phi_1/u \\
0 & 0 \\
\end{pmatrix} } = \overline{\beta}\begin{pmatrix}
\phi_1 \\
\phi_2
\end{pmatrix}N^{+} \cap H^p(\mathbb{D}, \mathbb{C}^2 ) $. The $\supseteq$ inclusion is clear. We will now show the $\subseteq$ inclusion. If we let $\begin{pmatrix}
F_1 \\
F_2
\end{pmatrix}$ lie in $ \ker T_{ \begin{pmatrix}
\phi_2/u & -\phi_1/u \\
0 & 0 \\
\end{pmatrix} }$, then $F_1\phi_2 = F_2 \phi_1$, and so $\overline{\beta}F_1\phi_2 = \overline{\beta} F_2 \phi_1$ and  $F_1$ can be written as $F_1 = \overline{\beta}\phi_1 \frac{F_2}{\overline{\beta}\phi_2}$. Furthermore $ \frac{F_2}{\overline{\beta}\phi_2}$ is in the Smirnov class, because $\overline{\beta}F_1\phi_2 = \overline{\beta}F_2\phi_1$ and $GCD(\overline{\beta} \phi_1, \overline{\beta}\phi_2) = 1$ so every inner factor of $\overline{\beta}\phi_2$ divides $F_2$. We can also write $F_2 = \overline{\beta}\phi_2 \frac{F_1}{\overline{\beta}\phi_1}$, and as $\overline{\beta}F_1\phi_2 = \overline{\beta} F_2 \phi_1$, we have $\frac{F_1}{\overline{\beta}\phi_1} = \frac{F_2}{\overline{\beta}\phi_2}$, so $\begin{pmatrix}
F_1 \\
F_2
\end{pmatrix} \in \overline{\beta}\begin{pmatrix}
\phi_1 \\
\phi_2
\end{pmatrix}N^{+} \cap H^p(\mathbb{D}, \mathbb{C}^2 ). $ 

Thus we have proved that if $\begin{pmatrix}
\phi_1 \\
\phi_2
\end{pmatrix},
\begin{pmatrix}
\psi_1 \\
\psi_2
\end{pmatrix} \in \ker T_H $ then $\ker T_{ \begin{pmatrix}
\phi_2/u & -\phi_1/u \\
0 & 0 \\
\end{pmatrix} } \subseteq \ker T_H .$
\end{proof}

\begin{prop}\label{inner cyclic}
Let $\theta$ be inner. Then $f \in N^+$ is cyclic for $B$ iff $\theta f$ is cyclic for $B$.
\end{prop}
\begin{proof}
If $f$ is not cyclic then it lies in a non-trivial $B$ invariant subspace. Then by Proposition \ref{bigger space} $f \in I^*(N^+)$ for some inner function $I$, which then means $\theta f \in ( \theta I)^*(N^+)$ and is therefore not cyclic for $B$. Conversely if $\theta f$ is not cyclic, $\theta f$ lies in some one component $B$ invariant subspace and hence so does $f$. So $f$ can not be cyclic.
\end{proof}
Combining the two previous Theorems and the previous Proposition we can deduce the following unifying theorem.
\begin{thm}
Let $\begin{pmatrix}
\phi_1 \\
\phi_2
\end{pmatrix},
\begin{pmatrix}
\psi_1 \\
\psi_2
\end{pmatrix} \in H^p(\mathbb{D}, \mathbb{C}^2 )$ be such that $\phi_1\psi_2 - \psi_1\phi_2 = 0$. Then we have

$$
\kappa_{min} \left( \begin{pmatrix}
\phi_1 \\
\phi_2
\end{pmatrix},
\begin{pmatrix}
\psi_1 \\
\psi_2
\end{pmatrix} \right) = \ker T_{ \begin{pmatrix}
\phi_2/u & -\phi_1/u \\
0 & \chi \\
\end{pmatrix} },
$$
where $u$ is an outer function such that $|u| = \phi_1 + \phi_2 +1$ and $\chi$ is our previously given symbol for the scalar Toeplitz kernel $\kappa_{min} ( \phi_2, \psi_2 )$. (Here if $\kappa_{min} ( \phi_2, \psi_2 ) = H^p$ the symbol is formally defined to be $0$.)
 \end{thm}
 
 \section*{Acknowledgements}
 The author is grateful to the EPSRC for financial support. \newline The author is grateful to Professor Partington for his valuable comments. \newline
 The author is grateful to the referee for their comments.
 \newpage
 \bibliographystyle{plain}
\bibliography{bibliography.bib}
\end{document}